\documentclass[a4paper,10pt]{amsproc}
\title{The ring of real-valued functions which are continuous on a dense cozero set}
\usepackage{amsfonts, amsmath, amssymb, graphicx, mathrsfs, geometry}
\usepackage[all]{xy}

\newdir{ >}{!/6pt/@{ }*:(1,0.2)@_{>}*:(1,-0.2)@^{>}}

\theoremstyle{plain}

\newtheorem{theorem}{Theorem}[section]
\newtheorem{lemma}[theorem]{Lemma}

\newtheorem{proposition}[theorem]{Proposition}
\theoremstyle{definition}
\newtheorem{definition}[theorem]{Definition}

\newtheorem{remark}[theorem]{Remark}
\newtheorem{counter example}[theorem]{Counter Example}

\newtheorem{corollary}[theorem]{Corollary}

\newtheorem{example}[theorem]{Example}

\numberwithin{equation}{section}

\author[A. Dey]{Amrita Dey}	\address{Department of Pure Mathematics, University of Calcutta, 35, Ballygunge Circular Road, Kolkata 700019, West Bengal, India}	\email{deyamrita0123@gmail.com}

\author[S. Bag]{Sagarmoy Bag}	\address{Department of Mathematics, Bankim Sardar College, Tangrakhali, West Bengal 743329, India}	\email{sagarmoy.bag01@gmail.com}

\author[D. Mandal]{Dhananjoy Mandal} \address{Department of Pure Mathematics, University of Calcutta, 35, Ballygunge Circular Road, Kolkata 700019, West Bengal, India}  \email{dmandal.cu@gmail.com, dmpm@caluniv.ac.in}


                                                                                            
\begin{document}
	
	\title[On $T''(X)$]{The ring of real-valued functions which are continuous on a dense cozero set}

	%

	\thanks {The first author is immensely grateful for the award of research fellowship provided by the University Grants Commission, New Delhi (NTA Ref. No. 221610014636).}

	\large
	\begin{abstract}
		Let $T''(X)$ and $T'(X)$ denote the collections of all real-valued functions on $X$ which are continuous on a dense cozero set and on an open dense subset of $X$ respectively. $T''(X)$ contains $C(X)$ and forms a subring of $T'(X)$ under pointwise addition and multiplication. We inquire when $T''(X)=C(X)$ and when $T''(X)=T'(X)$. We also ponder over the question when is $T''(X)$ isomorphic to $C(Y)$ for some topological space $Y$. We investigate some algebraic properties of the ring, $T''(X)$ for a Tychonoff space $X$. We provide several characterisations of $T''(X)$ as a Von-Neumann regular ring. We define nowhere almost $P$-spaces using the ring $T''(X)$ and characterise it as a Tychonoff space which has no non-isolated almost $P$-points. We show that a Tychonoff space with countable pseudocharacter is a nowhere almost $P$-space and highlight that this condition is not superflous using the closed ordinal space. 
	\end{abstract}	
		\maketitle
	
	\section{Introduction}
	
Let $X$ be a $T_1$ topological space and $C(X)$ denote the collection of all real-valued continuous functions on $X$. A set $U\subseteq X$ is said to be a cozero set in $X$ if $U=\{x\in X\colon f(x)\neq 0 \}$ for some $f\in C(X)$. We write $U=coz(f)$ in this case. The zero set of $f$, denoted by $Z(f)$, is the complement of $coz(f)$. Even though we use the notations $Z(f)=\{x\in X\colon f(x)=0\}$ and $coz(f)=X\setminus Z(f)$ for any function $f\in \mathbb{R}^X$, by a ``\textit{zero set}" (resp. ``\textit{cozero set}") we shall always mean that of a continuous function. Throughout this article we use the notation $\chi_A$ to denote the `characteristic function' of a subset, $A$ of $X$. That is, $\chi_A\colon X\longrightarrow \mathbb{R}$ is defined as follows: $\chi_A(x)=\begin{cases}
	1 &\text{ if }x\in A \\ 0 &\text{ otherwise}
\end{cases}.$	

A topological space $X$ is said to be an \textbf{almost $P$-space} \cite{Levy} if every non-empty $G_\delta$-set (equivalently, for a Tychonoff space, every non-empty zero set) has non-empty interior. Furthermore a point $p\in X$ is said to be an \textbf{almost $P$-point} if every non-empty $G_\delta$-set (equivalently, for a Tychonoff space, every non-empty zero set) containing $p$ has non-empty interior. Clearly $X$ is an almost $P$-space if and only if each point  in $X$ is an almost $P$-point. A topological space $X$ is said to be \textbf{perfectly normal} if it is normal and each closed set is a $G_\delta$-set in $X$. Equivalently a space $X$ is a perfectly normal space if each closed set in $X$ is a zero set in $X$. 

A ring $R$ is said to be an \textbf{almost regular ring}, or a \textbf{classical ring} if each element in $R$ is either a unit or a zero divisor. An ideal $I$ of a commutative ring $R$ with unity is a \textbf{$\boldsymbol{z}$-ideal} if for each $a\in I$, $M(a)\subseteq I$, where $M(a)$ denotes the intersection of all maximal ideals of $R$ containing $a$. A ring $S$ containing a reduced ring $R$ is called a \textbf{ring of quotients} of $R$ if and only if for each $s\in S\setminus \{0\}$, there exists $r\in R$ such that $sr\in R\setminus \{0\}$. An ideal $I$ of a ring $R$ is said to be its \textbf{minimal ideal} if it does not properly contain any non-trivial ideal. The \textbf{socle} of a ring $R$ is defined as the sum of its minimal ideals.

The ring $T'(X)$, consisting of real-valued functions on $X$ that are continuous on an open dense subset of $X$, was introduced by M. R. Ahmadi Zand \cite{A2010} in the year 2010. They established various algebraic properties of the ring $T'(X)$. Later on, in 2018, Gharabaghi et. al. \cite{GGT2018} studied some more properties of the ring.

This motivated us to collect all the real-valued functions on $X$ that are continuous on a dense cozero set in $X$, and we denote it as $T''(X)$.  In Section 2, we establish that  $T''(X)$ properly contains the ring $C(X)$ and contained in the ring $T'(X)$. We observe that for a Tychonoff space $X$, the ring $T''(X)$ coincides with $C(X)$ if and only if $X$ is an almost $P$-space and highlight that this result may fail if $X$ fails to be a Tychonoff space, using a suitable counterexample. We also discuss conditions under which $T''(X)$ is isomorphic to $C(Y)$ for some topological space $Y$.
In the literature it is observed that the ring $C(X)_F$, of real-valued functions on $X$ which are continuous every where on $X$ except on a finite set (\cite{GGT2018}), lies between $C(X)$ and $T'(X)$, that is, $C(X)\subseteq C(X)_F\subseteq T'(X)$.  In Section 2, we show that for a Tychonoff non-discrete almost $P$-space, $T''(X)$ is a proper subring of $C(X)_F$ and for a perfectly normal space, $X$ having an infinite nowhere dense closed set, $C(X)_F$ is a proper subring of $T''(X)$. 

We denote the collection $\{Z(f)\colon f\in T''(X) \}$ by $Z''[X]$ and as usual $Z[X]=\{ Z(f)\colon f\in C(X)\}$. Also, for an ideal $I$ of $T''(X)$, we denote the collection $\{Z(f)\colon f\in I \}$ by $Z[I]$ and for a collection $\mathfrak{F}\subseteq Z''[X]$, we denote the family $\{f\in T''(X)\colon Z(f)\in \mathfrak{F} \}$ by $Z^{-1}(\mathfrak{F})$. 

 At the end of Section 2, we introduce $z''$-filters and $z''$-ultrafilters on $X$. We investigate that for a proper ideal $I$ of $T''(X)$, $Z[I]$ is a $z''$-filter on $X$ and furthermore if $I$ is a maximal ideal, then $Z[I]$ is a $z''$-ultrafilter on $X$. Also shows that if $\mathfrak{F}$ is a $z''$-filter (resp. $z''$-ultrafilter) on $X$, $Z^{-1}(\mathfrak{F})$ is a proper ideal (resp. maximal ideal) of $T''(X)$.

In Section 3, we deal with certain algebraic properties of the ring $T''(X)$ and relate such properties with topological properties of the space $X$. We start by describing $z$-ideals of $T''(X)$ using zero sets of functions in $T''(X)$. Then we identify the non-zero minimal ideals of $T''(X)$ and use this to characterise the socle of $T''(X)$. We end this section by showing that for a Tychonoff space $X$, $T''(X)$ is a Noetherian ring (and/or Artinian ring) if and only if $X$ is finite.

A commutative ring $R$ with unity is said to be a \textbf{Von-Neumann regular ring} (or simply, a regular ring) if for each $a\in R$, there exists an $x\in R$ such that $a=a^2x$. It is well known that $T'(X)$ is always a (Von-Neumann) regular ring \cite{A2010}. We wonder if we can conclude the same for the ring $T''(X)$ or not. The answer (in general) is no. In Section 4, we discuss when is $T''(X)$ a regular ring. We realise that for a Tychonoff $P$-space, $T''(X)$ is a regular ring. But the converse of this observation is not true. Moreover, we get a large class of topological spaces, including non-trivial normed linear spaces, which are not $P$-spaces but for whom the rings of the form $T''(X)$ are regular rings. Also we show that if $X$ is a perfectly normal space, then $T''(X)$ is a regular ring and highlight that the converse of this result is not true either. We terminate this section with several characterisations of $T''(X)$ as a regular ring.

The functions of the form $\chi_{\{p\}}$, for some $p\in X$ have been useful in establishing various algebraic properties of rings such as $C(X)_F$ \cite{GGT2018}, $T'(X)$ \cite{GGT2018,A2010} and $B_1(X)$ \cite{2RM2019}. In this context, we call a topological space $X$ a \textbf{nowhere almost }$\boldsymbol{P}$\textbf{-space} if $\chi_{\{p\}}\in T''(X)$ for all $p\in X$. Section 5 of this article is devoted to study this class of topological spaces. We characterise a Tychonoff nowhere almost $P$-space topologically as a Tychonoff space $X$ which has no non-isolated almost $P$-point. It follows from this topological characterisation that no non-discrete Tychonoff space can be both a nowhere almost $P$-space and an almost $P$-space. A topological space $X$ is said to be of \textbf{countable pseudocharacter} if every singleton set is a $G_\delta$-set \cite{E1977}. Clearly, a perfectly normal space is of countable pseudocharacter. We observe that a Tychonoff space having countable pseudocharacter is a nowhere almost $P$-space and infer that perfectly normal spaces are nowhere almost $P$-spaces. We realise that the condition that `$X$ is of countable pseudocharacter' is not a redundant condition, with the help of the closed ordinal space. Furthermore we show that $C(X)_F\subseteq T''(X)$ if and only if $X$ is a nowhere almost $P$-space. We establish that for a nowhere almost $P$-space $X$, $T''(X)$ is an almost regular ring and disprove the converse with an example. We highlight that a nowhere almost $P$-space having atleast one non-isolated point cannot be an almost $P$-space. We conclude from this that for a nowhere almost $P$-space $X$, $C(X)=T''(X)$ if and only if $X$ is discrete. We end this section with a few algebraic properties of $T''(X)$ under the assumption that $X$ is a nowhere almost $P$-space.

We conclude this article by raising a few open problems for the reader.

\section{On $T''(X)$}  \label{sec2}


We initiate this section with the following theorem which follows from the facts that intersection of two open dense sets is a dense set and that the intersection of two cozero sets is also a cozero set.
\begin{theorem} \label{t2.1}
	$T''(X)$ is a ring under pointwise addition and multiplication.
\end{theorem}

It is easy to verify that  $C(X)\subseteq T''(X)\subseteq T'(X)$. \\

The following theorem is a  characterisation of an  almost $P$-space using this ring.
\begin{theorem} \label{t2.2}
	For a Tychonoff space $X$, $C(X)=T''(X)$ if and only if $X$ is an almost $P$-space.
\end{theorem}
\begin{proof}
	If $X$ is not an almost $P$-space, there exists an $f\in C(X)$ such that $Z(f)\neq \emptyset$ but $int$ $Z(f)=\emptyset$. Then $coz(f)$ is dense in $X$ and so $g=\chi_{coz(f)}$ is a member of $T''(X)$ but not of $C(X)$.
	
	Conversely, let $X$ be an almost $P$-space and $f\in T''(X)$. Then there exists $g\in C(X)$ such that $coz(g)$ is dense in $X$ and $f|_{coz(g)}$ is continuous. Since $coz(g)$ is dense in $X$, $int~ Z(g)=\emptyset$. It follows from our assumption that $Z(g)=\emptyset$ and so $f$ is continuous on $X(=coz(g))$.
\end{proof}

It is important to point out that the above result may fail if $X$ is not a Tychonoff space.

\begin{example} \label{eg1}
	Consider the topological space $X=\mathbb{R}$ equipped with the cofinite topology. It is clear that $\displaystyle{\mathbb{R}\setminus \mathbb{N}=\bigcap_{n\in \mathbb{N}} (\mathbb{R}\setminus \{n\})}$ is a $G_\delta$-set with empty interior. Thus $X$ is not an almost $P$-space. But $T''(X)=C(X)$ as the only cozero sets in $X$ are $\emptyset$ and $\mathbb{R}$.
\end{example}

We further wonder whether $T''(X)$ shall be isomorphic to $C(Y)$ for some topological space $Y$. Clearly, the answer is in affirmative for a Tychonoff almost $P$-space. For general case, we first recall that the ring $C(Y)$ is closed under uniform limit, for any topological space $Y$. Thus to show $T''(X)\cong C(Y)$ for some topological space $Y$, it first needs $T''(X)$ is to be closed under uniform limit. The following example shows that for a Tychonoff space $X$, $T''(X)$ is not isomorphic to $C(Y)$ for any topological space $Y$.

\begin{example} \label{eg0}
	Consider the real line, $\mathbb{R}$. Since $\mathbb{R}$ is a metric space, each closed set in $X$ is a zero set. This ensures that $T''(\mathbb{R})=T'(\mathbb{R})$. Suppose $\{r_n\}$ is an enumeration of rationals in $\mathbb{R}$ and define for each $m\in \mathbb{N}$, $f_m\colon \mathbb{R}\longrightarrow \mathbb{R}$ as follows: \[f_m(x)=\begin{cases}
		\frac{1}{n} &\text{ if }x=r_n,\, n\leq m\\
		0 &\text{ otherwise}
	\end{cases}. \] Then the sequence $\{f_n\}$ converges uniformly to $f\colon \mathbb{R}\longrightarrow \mathbb{R}$ where \[ f(x)=\begin{cases}
		\frac{1}{n} &\text{ if }x=r_n,\, n\in \mathbb{N}\\
		0 &\text{ otherwise}
	\end{cases}. \] Note that each $f_n$ is a member of $T''(\mathbb{R})$ but $f$ is not. Thus $T''(\mathbb{R})$ is not closed under uniform limit and hence $T''(\mathbb{R})$ is not isomorphic to $C(Y)$ for any topological space $Y$.
\end{example}

We use the above example as a motivation to establish the following theorem.

\begin{theorem} \label{t2.4}
	Let $X$ be a Tychonoff space such that it contains a countably infinite subset $N$ of $X$ satisfying the following conditions:
	\begin{enumerate}
		\item $N$ is not nowhere dense in $X$.
		\item $X\setminus N$ is dense in $X$.
		\item Each singleton subset of $N$ is a zero set in $X$.
	\end{enumerate}
	Then $T''(X)$ is not closed under uniform limit and hence it is not isomorphic to $C(Y)$ for any topological space $Y$. 
\end{theorem}
\begin{proof}
	Let $N=\{r_1,r_2,\cdots,r_n,\cdots \}$. Define for each $m\in \mathbb{N}$, $f_m\colon X\longrightarrow \mathbb{R}$ as follows: \[f_m(x)=\begin{cases}
		\frac{1}{n} &\text{ if }x=r_n,\, n\leq m\\
		0 &\text{ otherwise}
	\end{cases}. \] Then the sequence $\{f_n\}$ converges uniformly to $f\colon X\longrightarrow \mathbb{R}$ where \[ f(x)=\begin{cases}
		\frac{1}{n} &\text{ if }x=r_n,\, n\in \mathbb{N}\\
		0 &\text{ otherwise}
	\end{cases}. \] Since each singleton (and hence each finite) subset of $N$ is a zero set and $X\setminus N$ is dense in $X$, $f_m\in T''(X)$ for each $m\in \mathbb{N}$. 
	
	Now see that for each $r_k\in N$, any open neighbourhood of $r_k$ meets $X\setminus N$. Choose $\epsilon>0$ such that $0\notin (\frac{1}{k}-\epsilon,\frac{1}{k}+\epsilon)$. Then for any neighbourhood $U$ of $r_k$, $f(U)\nsubseteq (\frac{1}{k}-\epsilon,\frac{1}{k}+\epsilon)$. Thus $f$ is discontinuous at each point in $N$. Now suppose $c\in X\setminus N$, then $f(c)=0$. Let $\epsilon>0$. Then by Archimedean property it follows that there exists $k\in \mathbb{N}$ such that $0<\frac{1}{n}<\frac{1}{k}<\epsilon$. Let $U_1$ be any open neighbourhood of $c$ and $U=U_1\setminus \{r_1,r_2,\cdots r_{k-1} \}$. Then $U$ is a open neighbourhood of $c$ and $f(U)\subseteq (f(c)-\epsilon,f(c)+\epsilon)$. Thus $f$ is continuous at each point in $X\setminus N$. We claim that $f\notin T''(X)$. If possible let $f\in T''(X)$. Then there exists $h\in C(X)$ such that $f$ is continuous on $X\setminus Z(h)$ and $X\setminus Z(h)$ is dense in $X$. This implies that $X\setminus Z(h)\subseteq X\setminus N$ which implies that $N\subseteq Z(h)$ and hence $int\;\overline{N}\subseteq int\;Z(h)$. But $int\; Z(h)=\emptyset$ which contradicts that $N$ is not nowhere dense in $X$. Thus $f\notin T''(X)$.
\end{proof}

\begin{example} \label{eg6}
	Note that the third condition on $N$ in Theorem \ref{t2.4} is automatically satisfied if $X$ is of countable pseudocharacter, and hence for perfectly normal spaces. The entire hypothesis of Theorem \ref{t2.4} is true for the real line $\mathbb{R}$ equipped with the usual topology. Moreover, the hypothesis of the aforementioned theorem holds for any non-trivial separable normed linear space. Indeed if $N$ is a countably infinite dense subset of a normed linear space $X$, then it follows from the fact that any open ball in $X$ consists of uncountably many points, that $X\setminus N$ is also dense in $X$.  Moreover, the Sorgenfrey line, $\mathbb{R}_l$ is also satisfies the conditions mentioned in Theorem \ref{t2.4}, for $N=$ the set of all rational numbers, even though $\mathbb{R}_l$ is not a normed linear space. Thus $T''(X)$ is not isomorphic to $C(Y)$ for any topological space $Y$ for any of the above mentioned topological spaces.
	
\end{example}




\begin{remark}
	We observe that (as seen in Example \ref{eg0}), for $X=\mathbb{R}$, $T''(X)=T'(X)$. In fact, this is true for any perfectly normal space. This ensures that Theorem \ref{t2.4} is valid for the ring $T'(X)$ as well, if $X$ is considered to be a perfectly normal space.
\end{remark} 
 
 That $T''(X)$ is in general a proper subring of $T'(X)$ can be noted in the following example.

\begin{example} \label{eg2}
	Let $X$ be the one point compactification of $\mathbb{R}$, endowed with discrete topology. Then the characteristic function $\chi_{\{\infty\}}$ is continuous on $\mathbb{R}$, which is an open dense subset of $X$. However, $\{\infty\}$ is not a $G_\delta$-set and hence it is not a zero set. Also note that no proper subset of $\mathbb{R}$ is dense in $X$. Hence $\chi_{\{\infty\}}\notin T''(X)$.
\end{example}
\begin{remark}
	It is known that the one point compactification of $\mathbb{R}$, endowed with discrete topology, is not a perfectly normal space. Thus Example \ref{eg2} highlights that if $X$ fails to be a perfectly normal space, then $T''(X)\neq T'(X)$.
\end{remark}

Whether there exists a topological space outside the class of perfectly normal spaces for which $T''(X)=T'(X)$ remains an unanswered question.


We recall that, like $T''(X)$, the ring $C(X)_F$, of all real-valued functions on $X$ which are continuous everywhere on $X$ except on a finite set (\cite{GGT2018}), lies between $C(X)$ and $T'(X)$, that is, $C(X)\subseteq C(X)_F\subseteq T'(X)$. So naturally we ponder over the relationship between the rings $C(X)_F$ and $T''(X)$.

\begin{remark}
	From Theorem \ref{t2.2}, it follows that if $X$ is a non-discrete Tychonoff almost $P$-space, then $T''(X)=C(X)$ and so $T''(X)$ is a proper subring of $C(X)_F$.
	
	Also if $X$ is a perfectly normal space, then $T''(X)=T'(X)$ and so $C(X)_F\subseteq T''(X)$. Moreover if $X$ contains an infinite nowhere dense closed set, $C(X)_F$ is a proper subring of $T''(X)$. In particular for a non-trivial normed linear space $X$, $C(X)_F$ is a proper subring of $T''(X)$.
	
\end{remark}

We thus realise that no general comparison can be made between the two rings $C(X)_F$ and $T''(X)$. In this context we raise a question stated at the end of the article.

However, when is $C(X)_F$ a subring of $T''(X)$ has been answered in Theorem \ref{t4.2} of this article.

Just as in $C(X)$, we have the following characterisation of units in the ring $T''(X)$.

\begin{theorem} \label{t2.3}
	An element $f\in T''(X)$ is  a unit in $T''(X)$ if and only if $Z(f)= \emptyset$.
\end{theorem}
\begin{proof}
	Let $f\in T''(X)$ be such that $Z(f)= \emptyset$. Then there exists a dense cozero set $U$ in $X$ such that $f$ is continuous on $U$. Thus $g\colon X\longrightarrow \mathbb{R}$, defined as $g(x)=\frac{1}{f(x)}$ for all $x\in X$; is continuous on $U$ as well and so $g\in T''(X)$ with $fg=\boldsymbol{1}$. 
	
	The converse is obvious.
\end{proof}

The above theorem ensures that for a proper ideal $I$ of $T''(X)$, $Z[I]$ does not contain the empty set. This motivates us to check if the collection $Z[I]$ possesses filter-like properties.

We call a collection $\mathfrak{F}\subseteq Z''[X]$ a $\boldsymbol{z''}$\textbf{-filter} on $X$ if the following conditions hold: \begin{enumerate}
	\item $\emptyset\notin \mathfrak{F}$,
	\item If $Z\in \mathfrak{F}$ and $Z_1\in Z''[X]$ with $Z\subseteq Z_1$, then $Z_1\in \mathfrak{F}$ and
	\item $\mathfrak{F}$ is closed under finite intersections.
\end{enumerate}
We further call a $z''$-filter, $\mathfrak{F}$ on $X$ a $\boldsymbol{z''}$\textbf{-ultrafilter} on $X$ if it is not properly contained in any $z''$-filter on $X$.

The following proposition can be proved by closely following the proofs of Theorems 2.3 and 2.5 in \cite{GJ1976}.
\begin{proposition} \label{p2.1}
	The following statements are true for a topological space $X$.
	\begin{enumerate}
		\item If $I$ is a proper ideal of $T''(X)$, then $Z[I]$ is a $z''$-filter on $X$.
		\item If $\mathfrak{F}$ is a $z''$-filter on $X$, then $Z^{-1}(\mathfrak{F})$ is a proper ideal of $T''(X)$.
		\item If $M$ is a maximal ideal of $T''(X)$, then $Z[M]$ is a $z''$-ultrafilter on $X$.
		\item If $\mathfrak{U}$ is a $z''$-ultrafilter on $X$, then $Z^{-1}(\mathfrak{U})$ is a maximal ideal of $T''(X)$.
	\end{enumerate}
	
\end{proposition}


\section{Algebraic properties of $T''(X)$}

Our intent in this section is to relate topological properties of $X$ with certain algebraic properties of the ring $T''(X)$. We start by discussing the nature of $z$-ideals of $T''(X)$. It is evident from the definition of $z$-ideal of a ring $R$ that each maximal ideal of $R$ is a $z$-ideal of $R$. It is well known that an ideal $I$ of $C(X)$ is a $z$-ideal if and only if whenever $Z(f)=Z(g)$ with $f\in C(X)$ and $g\in I$, $f\in I$. Naturally we question if there can be a similar characterisation for $z$-ideals of $T''(X)$. We answer this question in the affirmative.

\begin{theorem} \label{t5.01}
	An ideal $I$ of $T''(X)$ is a $z$-ideal if and only if whenever $Z(f)=Z(g)$ with $f\in T''(X)$ and $g\in I$, $f\in I$.
\end{theorem}

To prove this theorem we need to rely on the following lemma.
\begin{lemma} \label{l5.0}
	If $f\in M(g)$, then $Z(g)\subseteq Z(f)$.
\end{lemma}
\begin{proof}
	If there exists $p\in Z(g)\setminus Z(f)$, the ideal $M_p=\{h\in T''(X)\colon h(p)=0 \}$ is a maximal ideal of $T''(X)$. It follows that $g\in M_p$ but $f\notin M_p$. Thus we have $f\notin M(g)$.
\end{proof}
We now proceed to the proof of Theorem \ref{t5.01}.
\begin{proof}[Proof of Theorem \ref{t5.01}]
	Suppose $I$ is a $z$-ideal of $T''(X)$ and $f\in T''(X)$, $g\in I$ with $Z(f)=Z(g)$. Let $M$ be an arbitrary maximal ideal of $T''(X)$ containing $g$. To show that $f\in I$, it is enough to show that $f\in M$. See that $Z^{-1}(Z[M])$ is an ideal of $T''(X)$ containing $M$. It follows from the maximality of $M$ that $Z^{-1}(Z[M])=M$. Since $Z(f)=Z(g)\in Z[M]$, $f\in Z^{-1}(Z[M])=M$.

	Conversely assume that whenever $Z(f)=Z(g)$ with $f\in T''(X)$ and $g\in I$, $f\in I$. Suppose $g\in I$ and $f\in M(g)$. Then it follows from Lemma \ref{l5.0} that $Z(g)\subseteq Z(f)$. Thus $Z(f)=Z(fg)$ and $fg\in I$. Therefore it follows from the hypothesis that $f\in I$ and hence $M(g)\subseteq I$.
\end{proof}

The following corollary is immediate.
\begin{corollary} \label{c5.2}
	An ideal $I$ of $T''(X)$ is a $z$-ideal if and only if whenever $Z(f)\supseteq Z(g)$ with $f\in T''(X)$ and $g\in I$ implies $f\in I$.
\end{corollary}

Next we discuss the structure of non-zero minimal ideals of the ring $T''(X)$. We are able to identify exactly the collection of minimal ideals of $T''(X)$, for a Tychonoff space $X$.
\begin{theorem} \label{t5.0}
	Let $X$ be a Tychonoff space. A non-zero minimal ideal, $I$ of $T''(X)$ is generated by an idempotent $e$ such that $X\setminus Z(e)$ contains exactly one element.
\end{theorem}

\begin{proof}
	Since $T''(X)$ is a reduced ring, $I$ is generated by an idempotent $e$.
	If possible let $X\setminus Z(e)$ contain two distinct elements $x$ and $y$ in $X$. There exists $f\in C(X)$ such that $G=X\setminus Z(f)$ is dense in $X$ and $e$ is continuous on $G$. To achieve a contradiction we split the proof into two cases.
	
	\begin{itemize}
		\item[Case 1:] Let $x,y\in G$. Since $X$ is Tychonoff space, there exists $g\in C(G)$ such that $g(Z(e|_G)\cup \{x\})=\{0\}$ and $g(y)=1$. Define $\widehat{g}\colon X\longrightarrow \mathbb{R}$ as \[\widehat{g}(t)=\begin{cases}
			g(t) &\text{ if }t\in G=X\setminus Z(f) \\
			0 &\text{ if }t\in Z(f)\cap Z(e) \\
			1 &\text{ otherwise}
		\end{cases}. \] Then $\widehat{g}$ is continuous on $G$ and hence $\widehat{g}\in T''(X)$. It follows that $\widehat{g}=e\widehat{g}\in I\setminus \{0\}$. Hence by the minimality of $I$, we get $I=\langle \widehat{g} \rangle$ and thus $e=\widehat{g}h$ for some $h\in T''(X)$ which contradicts that $e(x)$ is non-zero.
		
		\item[Case 2:] At least one of $x$ and $y$ is not in $G$. Without loss of generality assume that $x\notin G$, that is $x\in Z(f)$ where $f\in C(X)$ and $int\;Z(f)=\emptyset$. This ensures that $\chi_{\{x\}}\in T''(X)$ as $\chi_{\{x\}}$ is continuous on $G$. Furthermore $\chi_{\{x\}}\in I\setminus \{\boldsymbol{0}\}$. It follows that $\{\boldsymbol{0}\}\subsetneqq \langle \chi_{\{x\}}\rangle \subsetneqq I $ which contradicts the minimality of $I$.
	\end{itemize}
	
\end{proof}

When discussing minimal ideals, the concept of the socle often arises. Naturally we aim to describe the socle of the ring $T''(X)$. We recall that the socle, $Soc(C(X))$ of $C(X)$ consists of all such functions that are non-zero atmost on a finite set. The following theorem shows that the socle of $T''(X)$ also enjoys such a description for a Tychonoff space $X$.

\begin{theorem} \label{t5.2}
	For a Tychonoff space $X$, the socle $Soc(T''(X))$ of the ring $T''(X)$ consists of such functions in $T''(X)$ that are non-zero atmost on a finite set.
\end{theorem}
\begin{proof}
	Let $f\in Soc(T''(X))$. Then $\displaystyle{f=\sum_{i=1}^n f_i}$ for some $n\in \mathbb{N}$, where $f_i\in I_i$ and $I_i$ is a minimal ideal of $T''(X)$, for $i=1,2,\cdots,n$. By Theorem \ref{t5.0}, each $f_i$ is zero everywhere on $X$ except (atmost) a single point $x_i$, $i=1,2,\cdots, n$. Thus $f$ is zero everywhere on $X$ except (atmost) on the finite set $\{x_1,x_2,\cdots, x_n \}$.
	
	Conversely, let $f\in T''(X)$ be such that $X\setminus Z(f)=\{x_1,x_2,\cdots, x_n \}$. If  $x_i$ is an isolated point in $X$ for some $i\in \{1,2,\cdots, n \}$, it follows that $\chi_{\{x_i\}}\in C(X)\subseteq T''(X)$. Now let $x_k$ be a non-isolated point in $X$ for some $k\in \{1,2,\cdots,n\}$. From the definition of $T''(X)$, it follows that there exists $h\in C(X)$ such that $X\setminus Z(h)$ is dense in $X$ and $f$ is continuous on $X\setminus Z(h)$. Since $x_k$ is non-isolated in $X$, $f$ is not continuous on $x_k$. This ensures that $x_k\in Z(h)$ and therefore $\chi_{\{x_k\}}$ is continuous on $X\setminus Z(h)$, which is dense in $X$. Hence $\chi_{\{x_k\}}\in T''(X)$. Thus $\chi_{\{x_i\}}\in T''(X)$ for all $i=1,2,\cdots, n$ and the ideals of the form $\langle \chi_{\{x_i\}}\rangle$ are minimal ideals of $T''(X)$. Finally we have $\displaystyle{f=\sum_{i=1}^nf\chi_{\{x_i\}}\in Soc(T''(X)) }$.
\end{proof}

\begin{remark}
	It is important here to question the validity of  Theorem \ref{t5.0} and Theorem \ref{t5.2} when we drop the condition that $X$ is a Tychonoff space. It has been noted in Theorem \ref{t4.1} that the aforementioned theorems are true for a nowhere almost $P$-space. So this question can be reduced as follows: Does there exist a nowhere almost $P$-space which is not Tychonoff?
\end{remark}

We terminate this section with the discussion of $T''(X)$ as a Noetherian and/or Artinian ring.

\begin{theorem} \label{t5.1}
	Let $X$ be a Tychonoff space. 
	Then the following statements are equivalent.
	\begin{enumerate}
		\item $X$ is finite.
		\item $T''(X)$ is Noetherian.
		\item $T''(X)$ is Artinian.
	\end{enumerate}
\end{theorem}	

\begin{proof}
	If $X$ is finite, then we have $T''(X)=C(X)=\mathbb{R}^X$ which is isomorphic to $\mathbb{R}^n$, where $n$ denotes the cardinality of $X$. This ensures that $T''(X)$ has finitely many ideals and is both Artinian and Noetherian.
	
	It follows from \cite[Proposition 6.2]{Atiyah} that a commutative ring $R$ with unity is Noetherian if and only if each ideal of $R$ is finitely generated. We shall show that if $X$ is an infinite space, then $T''(X)$ has an ideal which is not finitely generated. 
	
	Since $X$ is an infinite Hausdorff space, there exists a countably infinite discrete subspace $Y$ of $X$. Let $1$ denote a fixed point in $Y$. Define $I=\{f\in T''(X)\colon f(1)=0 \text{ and }f(y)=0 \text{ for all but finitely many }y\in Y \}$. Let $f_1,f_2,\cdots, f_k$ be $k$ functions in $I$, where $k\in \mathbb{N}$. We assert that $I$ is not generated by these $k$ functions. Indeed $\displaystyle{\bigcap_{i=1}^k Z(f_i) }$ contains all but finitely many elements in $Y$ and $1\in \displaystyle{\bigcap_{i=1}^k Z(f_i) }$. Let $\displaystyle{y_0\in\bigcap_{i=1}^k Z(f_i) \setminus \{1\} }$. Then since $Y$ is a discrete subspace of $X$, there exists an open neighbourhood $U$ of $y_0$ in $X$ such that $U\cap Y=\{y_0\}$. By our hypothesis there exists $f\in T''(X)$ such that $f(X\setminus U)=\{0\}$ and $f(y_0)=1$. It is evident that $f\in I$. But $f_i(y_0)=0$ for all $i=1,2,\cdots, k$. This ensures that $f\notin \langle f_1,f_2,\cdots, f_n\rangle$. Thus $\langle f_1,f_2,\cdots, f_n\rangle \subsetneq I$.
	
	Finally, we recall that if a commutative ring $R$ with unity is Artinian, it is necessarily Noetherian \cite[Chapter 16, Theorem 3]{Dummit}. Thus if $X$ is an infinite space, then $T''(X)$ is not Artinian either.
\end{proof}

\section{When is $T''(X)$ a Von-Neumann Regular?}

Another ring property that has been of interest in the study of rings of functions is the (Von-Neumann) regularity of such rings. We explore this property of the ring $T''(X)$ in this  section to underscore its significance.

We first observe that for a Tychonoff $P$-space $X$, $T''(X)=C(X)$ (Since a $P$-space is necessarily an almost $P$-space) and $C(X)$ is a regular ring (See 4J \cite{GJ1976}). Thus the following inference is immediate.

\begin{theorem} \label{t3.0}
	For a Tychonoff $P$-space $X$, $T''(X)$ is a regular ring.
\end{theorem}

It has been observed by M. R. Ahmadi Zand in \cite{A2010} that $T'(X)$ is a regular ring for any topological space $X$. Thus, when $T'(X)=T''(X)$, $T''(X)$ then automatically becomes a regular ring. The following result is a direct consequence of the discussions preceding Example \ref{eg2}.

\begin{theorem} \label{t3.2}
	For a perfectly normal space $X$, $T''(X)$ is a regular ring.
\end{theorem}
The following examples hold significant importance. 
\begin{example} \label{eg3}
	Let $X$ be a perfectly normal space which is not a $P$-space ( for example, the real numbers set $\mathbb{R}$ equipped with the usual topology). Then it follows from Theorem \ref{t3.2} that $T''(X)$ is a regular ring.
	
	It is known that any normed linear space $X$ with $|X|\geq 2$  is a perfectly normal space which is not a $P$-space. Thus $T''(X)$ is regular.
\end{example}

\begin{example} \label{eg4}
	Let $X$ be an uncountable space with a distinguished point $p$ such that each $x\in X\setminus \{p\}$ is isolated in $X$ and $U$ is a neighbourhood of $p$ if and only if $X\setminus U$ is a countable set. Then $X$ is a $P$-space such that $\{p\}$ is a closed set which is not a zero set. This ensures that $X$ is not perfectly normal \cite[Exercise 4N]{GJ1976}. Hence by Theorem \ref{t3.0}, $T''(X)$ is a regular ring.
\end{example}

\begin{remark} \label{r2}
	Examples \ref{eg3} and \ref{eg4} ensure that the converse of Theorems \ref{t3.0} and \ref{t3.2} (respectively) may not be true in general.

\end{remark}

We highlight here that in the above examples the topological spaces  are either perfectly normal spaces or  $P$-spaces. Whether there exists a space $X$ which is neither a $P$-space nor a perfectly normal space, for which $T''(X)$ is a regular ring remains an undecided problem.

\begin{remark} \label{r5}
	It is important to mention at this point that $T''(X)$ is not always a regular ring. As a matter of fact, if $X$ is a Tychonoff space which is an almost $P$-space but not a $P$-space (For example, consider $X$ to be the one-point compactification of $\mathbb{N}$), then $T''(X)=C(X)$ which is not regular (See Exercise 4J \cite{GJ1976}).
\end{remark}

Having discussed some examples and various cases regarding the regularity of $T''(X)$, we now wish to characterise spaces for which $T''(X)$ is a regular ring.

It is well known that the ring $C(X)$ is regular if and only if each zero set in $X$ is clopen in $X$ \cite[Exercise 4J]{GJ1976}. Keeping this in mind, we achieve the following result.

\begin{theorem} \label{t3.1}
	$T''(X)$ is a regular ring if and only if for each $Z\in Z''[X]$, there exists a dense cozero set $U$ in $X$ such that $Z\cap U$ is clopen in $U$.
\end{theorem}
\begin{proof}
	Let $T''(X)$ be a regular ring and $Z\in Z''[X]$. Then there exists $f\in T''(X)$ such that $Z=Z(f)$. By our assumption, there exists $g\in T''(X)$ such that $f=f^2g$. Since $f,g\in T''(X)$, there exists a dense cozero set $U$ in $X$ such that $f,g$ are continuous on $U$ and $Z(f)\cap U$ is clopen in $U$ as $Z(f)=X\setminus Z(\boldsymbol{1}-fg)$.
	
	Conversely, let $f\in T''(X)$. Then by  our assumption, there exists a dense cozero set $U$ in $X$ such that $Z(f)\cap U$ is clopen in $U$. Define $g\colon X\longrightarrow \mathbb{R}$ as $\displaystyle{g(x)=\begin{cases}
			0 &\text{ if }x\in Z(f) \\
			\frac{1}{f(x)} &\text{ if }x\notin Z(f)
	\end{cases}}$. Then $g$ is continuous on $U$ and thus $g\in T''(X)$ with $f=f^2g$. Hence $T''(X)$ is  regular.
\end{proof}

For a Tychonoff space $X$, $C(X)$ is regular if and only if each prime ideal in $C(X)$ is maximal if and only if every principal ideal is generated by an idempotent. A number of such characterisations for $C(X)$ to be regular are provided by Gillman and Jerison in \cite[Exercise 4J]{GJ1976}. We provide analogous characterisations for $T''(X)$ to be a regular ring.

\begin{theorem} \label{t3.3}
	For a topological space $X$, the following statements are equivalent.
	\begin{enumerate}
		\item \label{p1} $T''(X)$ is a regular ring.
		\item \label{p2} Every prime ideal of $T''(X)$ is maximal.
		\item \label{p3} For $f\in T''(X)$, there exists $g\in T''(X)$ with $Z(f)=X\setminus Z(g)$.
		\item \label{p4} For each $Z\in Z''[X]$, there exists a dense cozero set $U$ in $X$ such that $Z\cap U$ is clopen in $U$.
		\item \label{p5} Every ideal of $T''(X)$ is a $z$-ideal of $T''(X)$.
		\item \label{p6} Every ideal of $T''(X)$ is an intersection of prime ideals in it.
		\item \label{p7} Every ideal of $T''(X)$ is an intersection of maximal ideals in it. 		
		\item \label{p8} For every $f,g\in T''(X)$, $\langle f,g\rangle=\langle f^2+g^2\rangle $.	
		
		\item \label{p9} Every principal ideal in $T''(X)$ is generated by an idempotent.
		
	\end{enumerate}
\end{theorem}
\begin{proof}
	It follows directly from \cite[Theorem 1.16]{Goodearl} and the fact that $T''(X)$ is a reduced ring that the first two assertions are equivalent. Also, the equivalence between the statements (\ref{p1}) and (\ref{p4}) has already been achieved in Theorem \ref{t3.1}.
	
	Suppose $T''(X)$ is a regular ring. Then for $f\in T''(X)$, there exists $g\in T''(X)$ such that $f=f^2g$ and we have $Z(f)=X\setminus Z(\boldsymbol{1}-fg)$ where $\boldsymbol{1}-fg\in T''(X)$. Furthermore if $I$ is an ideal of $T''(X)$ and $f,g\in T''(X)$ such that $f\in I$ with $Z(f)=Z(g)$. It follows that there exists $h\in T''(X)$ such that $Z(f)=X\setminus Z(\boldsymbol{1}-fh)$ and so $Z(g)=X\setminus Z(\boldsymbol{1}-fh)$. Thus we have $g(\boldsymbol{1}-fh)=\boldsymbol{0}$ which implies that $g=fgh\in I$. Therefore each ideal in $T''(X)$ is a $z$-ideal in it. To show that each ideal is the intersection of prime ideals in $T''(X)$, it is sufficient to show that an ideal $I$ of $T''(X)$ is the intersection of prime ideals in $T''(X)$ containing $I$. At this point we recall that in a commutative ring, $R$ with unity, if $I$ is an ideal of $R$, the intersection of prime ideals of $R$ containing $I$ is the collection $\{a\in R\colon a^n\in I \text{ for some }n\in \mathbb{N} \}$ \cite[Theorem 0.18]{GJ1976}. Since in a $z$-ideal $I$ of $T''(X)$, $f^n\in I$ if and only if $f\in I$ for any $n\in \mathbb{N}$, it follows that if $I$ is a $z$-ideal in $T''(X)$, $I=\{f\in T''(X)\colon f^n\in I \text{ for some }n\in \mathbb{N} \}$ and so $I$ is the intersection of prime ideals containing it. This ensures each ideal in $T''(X)$ is the intersection of prime ideals containing it. In fact, it follows from the equivalence between the first two statements that each ideal in $T''(X)$ is the intersection of maximal ideals. Now let $f,g\in T''(X)$. Then $\langle f^2+g^2\rangle \subseteq \langle f,g\rangle $. Since each ideal in $T''(X)$ is a $z$-ideal, in particular, $\langle f^2+g^2\rangle $ is a $z$-ideal. Since $Z(f^2+g^2)\subseteq Z(f)$ it follows from Corollary \ref{c5.2} that $f\in \langle f^2+g^2\rangle $. Similarly we conclude that $g\in \langle f^2+g^2\rangle $ and hence $\langle f,g\rangle \subseteq \langle f^2+g^2\rangle $. Thus for each pair $f,g$ of functions in $T''(X)$, $\langle f,g\rangle =\langle f^2+g^2\rangle $. Again let $f\in T''(X)$, our aim is to show that $\langle f\rangle $ is generated by an idempotent in $T''(X)$. There exists a $g\in T''(X)$ such that $f=f^2g$. It follows that $e=fg$ is the required idempotent in $T''(X)$. Thus (\ref{p1}) implies all other statements.
	
	Let us suppose that (\ref{p3}) holds. Then for $Z(f)\in Z''[X]$, for $f\in T''(X)$, there exists $g\in T''(X)$ such that $Z(f)=X\setminus Z(g)$. Then there exists a dense cozero set $U$ in $X$ such that $f$ and $g$ are continuous on $U$. Therefore $Z(f)\cap U=Z(f|_U)=U\setminus Z(g|_U)$ is clopen in $U$. This proves (\ref{p4}). Thus the first four statements are equivalent.
	
	Next we assume (\ref{p5}) and prove (\ref{p1}). Let $f\in T''(X)$. Then the ideal $\langle f^2\rangle $ is a $z$-ideal and since $Z(f)=Z(f^2)$, $f\in \langle f^2\rangle $ and hence there exists $g\in T''(X)$ such that $f=f^2g$.
	
	The statement (\ref{p6}) implies (\ref{p1}) can be seen as follows. Let $f\in T''(X)$. Then $I=\langle f^2\rangle $ equals the intersection of prime ideals in $T''(x)$ containing $I$. It follows from \cite[Theorem 0.18]{GJ1976} that $I=\{g\in T''(X)\colon g^n\in I\text{ for some }n\in \mathbb{N} \}$. In particular since $f^2\in I$, it follows that $f\in I$ and so there exists $g\in T''(X)$ such that $f=f^2g$.
	
	The equivalence between the statements (\ref{p6}) and (\ref{p7}) follows from the equivalence among the statements (\ref{p1}), (\ref{p2}) and (\ref{p6}) and the fact that a maximal ideal in $T''(X)$ is always prime.
	
	That (\ref{p8}) implies (\ref{p1}) follows by taking $f\in T''(X)\; and\;g=\boldsymbol{0}$ in statement (\ref{p8}).
	
	Finally to show (\ref{p9}) implies (\ref{p1}), let $f\in T''(X)$. Then there exists an idempotent $e\in T''(X)$ such that $\langle f\rangle =\langle e\rangle $. Thus there exist $g,h\in T''(X)$ such that $f=eg$ and $e=fh$. So $f=eg=e^2g$, as $e$ is an idempotent and $f=e^2g=f^2(h^2g)$. Thus $T''(X)$ is a regular ring. This completes the proof.
\end{proof}


\section{Nowhere almost $P$-Spaces} 

One can observe that the presence of the characteristic functions of singleton sets, in certain subrings of $\mathbb{R}^X$ (such as $C(X)_\mathcal{P}$, $\mathcal{M}_\circ(X,\mu)$ and $B_1(X)$) provides them with some interesting algebraic properties (See \cite{GGT2018, BAMRS2022, MBM2021, 2RM2019, DABM}). We observe that the ring $T''(X)$ does not, in general, contain all such characteristic functions. Indeed if $X$ is a non-discrete almost $P$-space, then there exists a non-isolated point $p\in X$. Now by Theorem \ref{t2.2}, $T''(X)=C(X)$ and so $\chi_{\{p\}}\notin C(X)=T''(X)$. For example, let $X=\mathbb{N}\cup \{\infty \}$ be the one point compactification of $\mathbb{N}$. Then $\chi_{\{\infty \}}\notin C(X)=T''(X)$. This motivates us to study topological spaces for which $T''(X)$ contains the characteristic functions $\chi_{\{x\}}$ for each $x\in X$.
\begin{definition}
We call a topological space $X$ a nowhere almost $P$-space if $\chi_{\{p\}}\in T''(X)$ for all $p\in X$.
\end{definition}

The significance of such a nomenclature is relevant from the next topological characterisation of nowhere almost $P$-spaces.

\begin{theorem} \label{t4.0}
A Tychonoff space $X$ is a nowhere almost $P$-space if and only if no non-isolated point in $X$ is an almost $P$-point.
\end{theorem}

\begin{proof}
Let $X$ be a Tychonoff nowhere almost $P$-space and $x\in X$ be a non-isolated point in $X$. By our hypothesis $\chi_{\{x\}}\in T''(X)$ and so there exists $f\in C(X)$ such that $X\setminus Z(f)$ is dense in $X$ and $\chi_{\{x\}}$ is continuous on $X\setminus Z(f)$. Since $x\in X$ is non-isolated, $\chi_{\{x\}}$ is not continuous at $x$. Thus $X\setminus Z(f)\subseteq X\setminus \{x\}$ and hence $x\in Z(f)$. However $X\setminus Z(f)$ is dense in $X$ implies that $int$ $Z(f)=\emptyset$. Thus $x$ is not an almost $P$-point.

Conversely, let each non-isolated point in a Tychonoff space $X$ be not an almost $P$-point and choose $x\in X$. If $x$ is an isolated point then $\chi_{\{x\}}\in C(X)\subseteq T''(X)$. Assume that $x$ is a non-isolated point in $X$. Then there exists $f\in C(X)$ such that $x\in Z(f)$ but $int$ $Z(f)=\emptyset$. This ensures that $X\setminus Z(f)$ is dense in $X$ and clearly $\chi_{\{x\}}$ is continuous on $X\setminus Z(f)$. Thus $\chi_{\{x\}}\in T''(X)$.
\end{proof}

The above characterisation of nowhere almost $P$-spaces indicates that the concept of a nowhere almost $P$-space is kind of a opposite (loosely speaking) of that of an almost $P$-space. Here we must highlight that isolated points are always almost $P$-points. This brings us to the following corollary which follows directly from Theorem \ref{t4.0}.

\begin{corollary} \label{c4.1}
A Tychonoff space $X$ is both a nowhere almost $P$-space as well as an almost $P$-space if and only if $X$ is a discrete space.
\end{corollary}

In order to make the study of nowhere almost $P$-spaces meaningful, we need some concrete examples of topological spaces which are in fact nowhere almost $P$-spaces. First we recall that a topological space $X$ is said to be of countable pseudocharacter if every singleton set is a $G_\delta$-set. Equivalently, a Tychonoff space $X$ is of countable pseudocharacter if every singleton set is a zero set.

We percieve that Tychonoff spaces with countable pseudocharacter provide us with a large class of nowhere almost $P$-spaces. This is evident from the following proposition and example.

\begin{proposition} \label{p4.1}
A Tychonoff space $X$ with countable pseudocharacter is a nowhere almost $P$-space. 
\end{proposition}
\begin{proof}
Let us assume that $X$ is a Tychonoff space with countable pseudocharacter and $p\in X$. Then $\{p\}$ is a zero set in $X$. If $p$ is an isolated point, then $\chi_{\{p\}}\in C(X)\subseteq T''(X)$. Otherwise $\chi_{\{p\}}$ is continuous on $X\setminus \{p\}$ which is a dense cozero set in $X$, and so $\chi_{\{p\}}\in T''(X)$.
\end{proof}

\begin{example}
Since a perfectly normal $T_1$-space $X$ is of countable pseudocharacter, it follows that perfectly normal $T_1$-spaces are nowhere almost $P$-spaces. 

In particular, all metric spaces are nowhere almost $P$-spaces.

The space $\omega_1$ of all countable ordinals is an example of a Tychonoff space with countable pseudocharacter which is not perfectly normal.
\end{example}

We must highlight that the condition `$X$ has countable pseudocharacter' is not superflous in Proposition \ref{p4.1}. The following example illustrates this with clarity.

\begin{example} \label{eg5}
Let $\omega_1$ be the first uncountable ordinal and hence the space of all countable ordinals. We now show that the closed ordinal space, $X=\omega_1\cup \{\omega_1 \}(\equiv \omega_1+1)$ is not of countable pseudocharacter and is in turn not a nowhere almost $P$-space. 

We claim that $\{\omega_1\}$ is not a $G_\delta$-set in $X$. If not, then there exists $\{\alpha_n\colon n\in \mathbb{N} \}\subseteq \omega_1$ such that $\displaystyle{\{\omega_1 \}=\bigcap_{n\in \mathbb{N}}(\alpha_n,\omega_1]} $. Let $\alpha=\sup \{\alpha_n\colon n\in \mathbb{N} \}$. Then $\alpha<\omega_1$, since each $\alpha_n$ is countable and the countable union of countable sets is countable. Thus $\displaystyle{ (\alpha, \omega_1] \subseteq \bigcap_{n\in \mathbb{N}}(\alpha_n,\omega_1]}=\{\omega_1 \}$ which is a contradiction. Thus $X$ is not of countable pseudocharacter.

Again we claim that $\chi_{\{\omega_1 \}}\notin T''(X)$. If not then there exists a dense cozero set $Y\subseteq X\setminus \{\omega_1 \}$. Since $Y$ is dense in $X$, $X\setminus Y$ consists only of limit ordinals and $\omega_1 \in X\setminus Y$. Since $Y$ is a cozero set in $X$, there exists $f\in C(X)$ such that $Z(f)=X\setminus Y$. Therefore there exists $\alpha<\omega_1$ such that $f(x)=0$ for all $x\geq \alpha$. In particular, $f(\alpha+1)=0$. Thus $\alpha+1$ is a non-limit ordinal in $Z(f)$, which is a contradiction.
\end{example}

\begin{remark} \label{r4}
The fact that the normal space $X=\omega_1+1$ in Example \ref{eg5} is not a nowhere almost $P$-space can also be deduced using Theorem \ref{t4.0}. See that if $\omega_1\in Z(f)$ for some $f\in C(X)$, there exists an $\alpha<\omega_1$ such that $f(x)=0$ for all $x\geq \alpha$. (See \cite{GJ1976}) Thus $\{x\in X\colon x\geq \alpha \}\subseteq Z(f)$ and so $int$ $Z(f)\neq \emptyset$. Thus $\omega_1$ is a non-isolated almost $P$-point in $X$.
\end{remark}

We have certainly pondered whether the converse of Proposition \ref{p4.1} holds true. We have however neither been able to prove nor disprove the converse. This question is stated at the end of this article and is left for the readers.

A question was raised in Section 2: When is $C(X)_F$ a subring of $T''(X)$? We have established that such spaces are nothing but nowhere almost $P$-spaces.

\begin{theorem} \label{t4.2}
$C(X)_F$ is a subring of $T''(X)$ if and only if $X$ is a nowhere almost $P$-space.
\end{theorem}
\begin{proof}
Suppose $X$ is a nowhere almost $P$-space and $f\in C(X)_F$. Then the discontinuity set $D_f$ of $f$, is a finite set. Let $D_f=\{x_1,x_2,\cdots, x_n \}$, where $n\in \mathbb{N}$. Since $X$ is a nowhere almost $P$-space, $\chi_{\{x_i\}}\in T''(X)$ for all $i=1,2,\cdots,n$. Thus $\displaystyle{\chi_{_{D_f}}=\sum_{i=1}^{n}\chi_{\{x_i\}} \in T''(X)}$. Therefore there exists a dense cozero set $G$ in $X$ such that $\chi_{_{D_f}}$ is continuous on $G$. Note that each $x_i$ is a non-isolated point for $i=1,2,\cdots,n$. So $\chi_{_{D_f}}$ is discontinuous on $D_f$. This ensures that $G\subseteq X\setminus D_f$ and it follows that $f$ is continuous on $G$. Hence $f\in T''(X)$. The converse is obvious.
\end{proof}

\begin{remark} \label{r1}
We observe that if $X$ a nowhere almost $P$-space $X$, $T''(X)$ is an almost regular ring. To see this, let $f$ be a non-unit element in $T''(X)$. Then by Theorem \ref{t2.3}, there exists $p\in Z(f)$. Since $X$ is a nowhere almost $P$-space, $\chi_{\{p\}}\in T''(X)\setminus \{\boldsymbol{0} \}$ and we have $f\chi_{\{p\}}=\boldsymbol{0}$.

However, the converse of the above statement if false. Indeed if $X$  is an almost $P$-space having atleast one non-isolated point $p$ then by Theorem \ref{t2.2}, $C(X)=T''(X)$ but $\chi_{\{p\}}\notin C(X)$. Note that in this case $C(X)$ (and hence $T''(X)$) is an almost regular ring (see \cite{Levy}) even though $\chi_{\{p\}}\notin T''(X)$.
\end{remark}

Combining Theorem \ref{t2.2} and Corollary \ref{c4.1}, we can conclude that for a nowhere almost $P$-space, $C(X)=T''(X)$ if and only if $X$ is a discrete space. In fact, we get something more. The result below can be achieved by closely following the proof of Proposition 4.4 in \cite{DABM}.

\begin{proposition} \label{p4.2}
The following statements are equivalent for a nowhere almost $P$-space $X$.
\begin{enumerate}
	\item  $C(X)=T''(X)$.
	\item  $X$ is discrete.
	\item  $T''(X)$ is a ring of quotients of $C(X)$.
\end{enumerate} 
\end{proposition}

The next remark is immediate.
\begin{remark}
Proposition \ref{p4.2} is true for any Tychonoff space with countable pseudocharacter. This ensures that Proposition \ref{p4.2} is true for all perfectly normal spaces (and hence for all metric spaces).
\end{remark}

We discuss certain algebraic properties of the ring $T''(X)$ for a nowhere almost $P$-space $X$. The following results can be proved by closely tracing the steps used to prove Theorem 4.7 and Proposition 4.8 in \cite{DABM}.

\begin{theorem} \label{t4.1}
The following assertions are true for a nowhere almost $P$-space $X$.
\begin{enumerate}
	\item  A non zero ideal $I$ of $T''(X)$ is minimal if and only if there exists an $\alpha \in X$ such that $I=\langle \chi_{\{\alpha \}}\rangle $ if and only if $|\{ Z(f)\colon f\in I \}|=2$.
	\item  The socle of $T''(X)$ consists of all functions that vanish everywhere except on a finite set.
	\item $Soc(T''(X))=T''(X)$ if and only if $X$ is finite.
\end{enumerate}
\end{theorem}

\begin{remark} \label{r3}
It has been discussed above that a Tychonoff space with countable pseudocharacter is a nowhere almost $P$-space and a perfectly normal space is of countable pseudocharacter. We can thus conclude that Theorem \ref{t4.1} is valid for any Tychonoff space having countable pseudocharacter. In particular, the aforementioned theorem is valid for any perfectly normal space, and hence for any metric space.
\end{remark}

\section{Some open questions}

\begin{enumerate}
	
\item Can one characterise the zero divisors of the ring $T''(X)$ using the topology on $X$?
\item In light of the discussions made in Section \ref{sec2}, it remains a question as to when the ring $T''(X)$ is isomorphic to $C(Y)$ for some topologcal space $Y$?
\item We have noted that if $X$ is a perfectly normal space, then $T''(X)=T'(X)$ and we can conclude from Example \ref{eg2} that the condition  ``$X$ is perfectly normal'' is not superflous. The question remains whether there exists a topological space $X$ which is not perfectly normal but $T'(X)=T''(X)$.
\item In Theorem \ref{t4.2} we discussed when is $C(X)_F$ a subring of $T''(X)$. When is $T''(X)$ a subring of $C(X)_F$ and when is $C(X)_F=T''(X)$ are still unanswered questions.
\item Example \ref{eg6} deals with topological spaces that are all perfectly normal. We wonder if there exists a Tychonoff space $X$  which is not perfectly normal, yet satisfies the hypothesis of Theorem \ref{t2.4}?
\item Taking note of Theorems \ref{t3.0} and \ref{t3.2}, we pose the question if there exists a topological space $X$ which is neither a $P$-space nor a perfectly normal space, yet $T''(X)$ is Von-Neumann regular?
\item Is the converse of Proposition \ref{p4.1} true? In other words, does there exist a Tychonoff nowhere almost $P$-space which is not of countable pseudocharacter?
\item See that for a perfectly normal space $X$, $T''(X)=T'(X)=C(X)_\mathcal{P}$, where $\mathcal{P}$ is the collection of all closed nowhere dense sets and for an almost $P$-space $X$, $T''(X)=C(X)=C(X)_\mathcal{P}$, where $\mathcal{P}=\{\emptyset \}$. We speculate whether $T''(X)$ can always be expressed as $C(X)_\mathcal{P}$ for some ideal $\mathcal{P}$ of closed subsets of $X$. 
\item When $X$ is perfectly normal, $T''(X)=T'(X)$ which is a regular ring, and hence an almost regular ring as well. Furthermore, when $X$ is a nowhere almost $P$-space, we have seen that $T''(X)$ is almost regular. Moreover, when $X$ is an almost $P$-space, $T''(X)=C(X)$ which is an almost regular ring. So we raise the question: Does there exist a topological space $X$ for which $T''(X)$ is not an almost regular ring? 
\end{enumerate}

\section*{Acknowledgment}
The first author is immensely grateful for the award of research fellowship provided by the University Grants Commission, New Delhi (NTA Ref. No. 221610014636).

\end{document}